\documentclass[letterpaper,11 pt]{article}
\usepackage{graphicx}
\usepackage{amsfonts,amsmath,fullpage,bbm}
\usepackage{amssymb,amsthm,multirow,verbatim}
\usepackage{acronym,wrapfig,plain,mathrsfs,enumerate,relsize,color}
\newtheorem{algorithm}{Algorithm}
\usepackage{algorithm}
\newtheorem{theorem}{Theorem}
\newtheorem{corollary}{Corollary}
\newtheorem{lemma}{Lemma}

\newtheorem{remark}{Remark}

\newtheorem{assumption}{Assumption}
\newtheorem{definition}{Definition}
\usepackage{subfig}
\usepackage{algorithmic}
\usepackage{pifont}
\usepackage{footmisc}

\allowdisplaybreaks
\newcommand{\aj}[1]{{\color{black}#1}}
\newcommand{\mb}[1]{{\color{black}#1}}
\newcommand{\af}[1]{{\color{black}#1}}
\begin{document}
\allowdisplaybreaks
\title{Inexact-Proximal Accelerated Gradient Method for Stochastic Nonconvex Constrained Optimization Problems
}


\author{Morteza Boroun, Afrooz Jalilzadeh\footnote{ Department of Systems and Industrial Engineering, The University of Arizona, 1127 E James Rogers Way, Tucson, AZ. Email: morteza@email.arizona.edu, afrooz@arizona.edu.}}

\date{}

\maketitle
\begin{abstract}
Stochastic nonconvex optimization problems with nonlinear constraints have a broad range of applications in intelligent transportation, cyber-security, and smart grids. In this paper, first, we propose an inexact-proximal accelerated gradient  method to solve a nonconvex stochastic composite optimization problem where the objective is the sum of smooth and nonsmooth functions, the constraint functions are assumed to be deterministic and the solution to the proximal map of the nonsmooth part is calculated inexactly at each iteration. We demonstrate an asymptotic sublinear rate of convergence for stochastic settings using increasing sample-size considering the error in the proximal operator diminishes at an appropriate rate. Then we customize the proposed method for solving stochastic nonconvex optimization problems with nonlinear constraints and demonstrate a convergence rate guarantee. Numerical results show the effectiveness of the proposed algorithm.
\end{abstract}
\section{INTRODUCTION}
\label{sec:intro}

There is a rapid growth in the global urban population and the concept of smart cities is proposed to manage the impact of this surge in urbanization. Intelligent transportation, cyber-security, and smart grids are playing vital roles in smart city projects which are highly influenced by big data analytic and effective use of machine learning techniques \cite{ullah2020applications}. As data gets more complex and applications of machine learning algorithms for decision-making broaden and diversify, recent research has been shifted to constrained optimization problems with nonconvex objectives \cite{ma2017demand} to improve efficiency and scalability in smart city projects.

Consider the following constrained optimization problem with a stochastic and nonconvex objective: 
\begin{align}\label{p0}
\nonumber \min_{x\in X} \quad &f(x)\triangleq \mathbb E[F(x,\zeta(\omega))]\\
 \mbox{s.t.} \quad &\phi_i(x)\leq 0,\quad i=1,\hdots,m,
\end{align} 
where $ {\zeta}: \Omega \rightarrow
\mathbb{R}^o$, ${F}: \mathbb{R}^n \times \mathbb{R}^o \rightarrow
\mathbb{R}$, {and} $(\Omega,\mathcal{F},\mathbb{P})$ denotes the associated
probability space. We consider function $f(x): \mathbb R^n\to \mathbb R$ is smooth and possibly nonconvex, $\phi_i(x):\mathbb R^n\to \mathbb R$ are \af{deterministic,} convex, and smooth for all $i$, and set $X$ is convex and compact. To solve this problem, first we propose an algorithm for solving the following composite optimization problem  
\begin{align}\label{p1}
&\min_{x\in \mathbb R^n} g(x)\triangleq f(x)+h(x),
\end{align} 
where $h(x): \mathbb R^n\to \mathbb R$ is a convex function and possibly nonsmooth. 
Using the indicator function $\mathbb{I}_\Theta(\cdot)$, where $\mathbb{I}_\Theta(x)=0$ if $x\in \Theta$ and $\mathbb{I}_\Theta(x)=+\infty$ if $x\notin \Theta$, one can write problem \eqref{p0} in the form of problem \eqref{p1} by choosing 
 $h(x)= \mathbb{I}_\Theta(x)$ and $\Theta=\{x\mid x\in X, \ \phi_i(x)\leq 0, \ \forall i=1,\hdots,m\}$. Moreover, we show that how to customize the  proposed method to solve problem \eqref{p0}.  Indeed, proximal-gradient methods are an appealing approach for solving \eqref{p1} due to their computational efficiency and fast theoretical convergence guarantee. In deterministic and convex regime, subgradient methods have been shown to have a convergence rate of $\mathcal O(1/\sqrt{T})$, however, proximal-gradient methods can achieve a faster rate of \af{$\mathcal O(1/ T)$,  \af{where $T$ is the total number of iterations}}. Each iteration of a proximal-gradient method requires solving the following:
\begin{align}\label{prox}
\mbox{prox}_{\gamma,h}(y)=\underset{u\in \mathbb R^n}{\mbox{argmin}}\{h(u)+{1\over 2\gamma}\|u-y\|^2\}.
\end{align}
In many scenarios, computing the exact solution of the proximal operator may be expensive or may not have an analytic solution. In this work, we propose a gradient-based scheme to solve the nonconvex optimization problem \eqref{p1} by computing the proximal operator inexactly at each iteration. 

Next, we introduce important notations that we use throughout the paper and then briefly summarize the related research.  
\subsection{Notations}
We denote the optimal objective
value (or solution) of  \eqref{p1} by $g^*$ (or $x^*$) and the set of the optimal
solutions by $X^*$, which is assumed to be nonempty. For any $a\in \mathbb R$, we define $[a]_+=\max\{0,a\}$. $\mathbb{E}[\bullet]$ denotes the expectation with respect to
the probability measure $\mathbb{P}$ and $\mathcal B(s)=\{x\in \mathbb R^n \mid \|x\|\leq s\}$.  $\Pi_\Theta(\cdot)$ denotes the projection onto convex set $\Theta$ and $\mbox{\bf relint}(X)$ denotes the relative interior of the set $X$. Throughout the paper, $\mathcal{\tilde O}$ is used to suppress all the logarithmic terms. 
\subsection{Related Works}
There has been a lot of studies on first-order methods for convex optimization with convex constraints, see \cite{tran2014primal,xu2019iteration} for deterministic constraints and \cite{basu2019optimal,lan2016algorithms} for stochastic constraints. Nonconvex optimization problems without constraints or with easy-to-compute projection on the constraint set  have been studied by \cite{ghadimi2013stochastic,zhang2018convergence,lan2019accelerated}. When the function $f$ in problem \eqref{p1} is convex and $h$ is a nonsmooth function, \cite{schmidt2011convergence} showed that even with errors in the computation of the gradient and the proximal operator, the  inexact proximal-gradient method achieves the same convergence rates as the exact counterpart, if the magnitude of the errors is controlled in an appropriate rate. In nonconvex setting, assuming the proximal operator has an exact solution, \cite{ghadimi2016accelerated} obtained a convergence rate of $\mathcal O(1/T)$, using accelerated gradient scheme for deterministic problems and in stochastic regime using increasing sample-size they obtained the same convergence rate. Inspired by these two works, we present accelerated inexact proximal-gradient framework that can solve problems \eqref{p0} and \eqref{p1}.  {In deterministic regime,} \cite{kong2019complexity} analyzed the iteration-complexity of a quadratic penalty accelerated inexact proximal point method for solving linearly constrained nonconvex composite programs with iteration complexity of $ \mathcal {\tilde O}(\epsilon^{-3})$. Inexact proximal-point penalty method introduced by \cite{lin2019inexact} and \cite{li2021rate} can solve nonlinear constraints with complexity of $\mathcal{\tilde O}(\epsilon^{-2.5})$ and $ \mathcal {\tilde O}(\epsilon^{-3})$ {for affine equality constraints} and nonconvex constraints, respectively. {Recently, \cite{li2020augmented} showed complexity result of $\mathcal{\tilde O}(\epsilon^{-2.5})$ for deterministic problems with nonconvex objective and convex constraints with nonlinear functions to achieve $\epsilon$-KKT point.} 
{In stochastic regime,} \cite{boob2019stochastic} has studied functional constrained optimization problems and obtained a non-asymptotic convergence rate of $\mathcal{ O}(\epsilon^{-2})$ for stochastic problems with convex constraints to achieve $\epsilon^2$-KKT point. {In this paper, we obtain the same convergence rate under weaker assumptions. In particular, in contrast to  \cite{boob2019stochastic}, our analysis does not require the objective function to be Lipschitz and  we prove an asymptotic convergence rate result.}
Next, we outline the contributions of our paper. 

\subsection{Contributions}
In this paper, we consider a stochastic nonconvex  optimization problem with convex nonlinear constraints. We propose an inexact proximal accelerated gradient (IPAG) method where at each iteration the projection onto the nonlinear constraints is solved inexactly. By improving the accuracy of the approximate solution of the proximal subproblem (projection step) at an appropriate rate and ensuring feasibility at each iteration combined with a variance reduction technique, we demonstrate a convergence rate of $\mathcal O(1/T)$, where $T$ is the total number of iterations, and the oracle complexity (number of sample gradients) of $\mathcal O(1/\epsilon^2)$ to achieve an $\epsilon$-first-order optimality of problem \eqref{p0}. 
\af{To accomplish this task, first we} analyze the proposed method for the composite optimization problem \eqref{p1} which can be specialized to \eqref{p0} using an indicator function.  
Moreover, 
our proposed method requires weaker assumptions compare to  \cite{boob2019stochastic}.

Next, we state the main definitions and assumptions that we need for the convergence analysis. In Section \ref{conv},  we introduce the IPAG algorithm to solve the composite optimization problem and then in Section \ref{const opt} we show that IPAG method can be customized to solve a nonconvex stochastic optimization problem with nonlinear constraints \eqref{p0}.  Finally, in  section \ref{numer} we present some empirical experiments to show the benefit of our proposed scheme in comparison with a competitive scheme.

\subsection{Assumptions and Definitions}\label{asd}
Let $\rho$ be the error in the calculation of the proximal objective function achieved by $\tilde x$, i.e., 
\begin{align}\label{prox error}
{1\over 2\gamma}\|\tilde x-y\|^2+h(\tilde x)\leq \rho+\min_{x\in \mathbb R^n}\left\{{1\over 2\gamma}\|x-y\|^2+h(x)\right\},
\end{align}
and we call $\tilde x$ a $\rho$-approximate solution to the proximal problem. 
Next, we define $\rho$-subdifferential and then we state a lemma to characterize the elements of the $\rho$-subdifferential of $h$ at $x$.
\begin{definition}[$\rho$-subdifferential]\label{subgrad}
Given a convex function $h(x):\mathbb R^n\to \mathbb R$ and a positive scalar $\rho$, the $\rho$-approximate subdifferential of $h(x)$ at a point $x\in \mathbb R^n$, denoted as $\partial_\rho h(x)$, is
$$\partial_\rho h(x)=\{d\in \mathbb R^n: h(y)\geq h(x)+\langle d,y-x\rangle-\rho\}.$$
Therefore, when $d\in\partial_\rho h(x)$, we say that $d$ is a $\rho$-subgradient of $h(x)$ at point $x$.
\end{definition}
\begin{lemma}\label{error subdef} 
If $\tilde x$ is a \af{$\rho$-approximate} solution to the proximal problem \eqref{prox} in the sense of \eqref{prox error}, then there exists $v$ such that $\|v\|\leq \sqrt{2\gamma\rho}$ and
$$\tfrac{1}{\gamma}\left(y-\tilde x-v\right)\in \partial_{\rho} h(\tilde x).$$
\end{lemma}
Proof of Lemma \ref{error subdef} can be found in \cite{schmidt2011convergence}. Throughout the paper, we exploit the following basic lemma.  
\begin{lemma}\label{3norms}
Given a symmetric positive definite matrix $Q$, we have the following for any $\nu_1,\nu_2,\nu_3$:
\begin{align*}(\nu_2-\nu_1)^TQ(\nu_3-\nu_1)={1\over 2}(\|\nu_2-\nu_1\|^2_Q+\|\nu_3-\nu_1\|^2_Q-\|\nu_2-\nu_3\|^2_Q), \mbox{ where } \|\nu\|_Q \triangleq \sqrt{\nu^TQ\nu}.\end{align*}
\end{lemma}
In our analysis we use the following lemma \cite{ghadimi2016accelerated}. 
\begin{lemma}\label{bound gamma}
Given a positive sequence $\alpha_k$, define $\Gamma_k=1$ for $k=1$ and $\Gamma_k= (1-\alpha_k)\Gamma_{k-1}$ for $k>1$. Suppose a sequence $\{\chi_k\}_k$ satisfies $\chi_k\leq (1-\alpha_k)\chi_{k-1}+\lambda_k$, where $\lambda_k>0$. Then for any $k\geq1$, we have that $\chi_k\leq \Gamma_k\sum_{j=1}^k \gamma_j/\Gamma_j$.  
\end{lemma}
The following assumptions are made throughout the paper.
\begin{assumption}\label{assump1}  The following statements hold:
\begin{itemize}
\item[(i)] A slater point of problem \eqref{p0} is available, i.e.,  there exists  $x^{\circ}\in \mathbb R^n$ such that $\phi_i(x^{\circ})<0$ for all $i=1,\hdots,m$ and $x^{\circ}\in \mbox{\bf relint}(X)$. 
\item[(ii)] Function $f$ is smooth and  weakly-convex with Lipschitz continuous gradient, i.e. there exists $L,\ell\geq0$ such that 
$-\tfrac{\ell}{2}\|y-x\|^2\leq f(x)-f(y)-\langle \nabla f(x),y-x\rangle \leq \tfrac{L}{2} \|y-x\|^2$.

\item[(iii)] There exists $C>0$ such that $\|\mbox{prox}_{\gamma,h}(y)\|\leq C$ for any $\gamma>0$ and $y\in \mathbb R^n$.

\item[(iv)]   $\mathbb{E}[ \xi_k \mid \mathcal{F}_k] = 0$ holds a.s., where $ \xi_k \triangleq \nabla f(z_k) - \nabla F(z_k,\omega_k)$. Also, there exists $\tau>0$ such that {$\mathbb{E}[\|\bar  \xi_k\|^2\mid \mathcal{F}_k] \leq {\tau^2\over N_k}$} holds a.s.  for all $k$ and  $\mathcal{F}_k \triangleq \sigma\left(\{z_0,\bar \xi_0, z_1,\bar \xi_1 \hdots, z_{k-1},\bar\xi_{k-1}\}\right)$, where $ \bar \xi_k\triangleq \frac{\sum_{j=1}^{N_k} \nabla f(z_k)-\nabla F(z_k,\omega_{j,k})}{N_k}$.
\end{itemize}
\end{assumption}
\af{Note that Assumption 1 is a common assumption in nonconvex and stochastic optimization problems and it holds for many real-world problems such as problem of non-negative principal component analysis and classification problem with nonconvex loss functions \cite{pham2020proxsarah}.}
\section{CONVERGENCE ANALYSIS}\label{conv}
In this section, we propose an inexact-proximal accelerated gradient scheme for solving problem \eqref{p1} assuming that an inexact solution to the proximal subproblem exists through an inner algorithm $\mathcal M$. Later in section \ref{const opt}, we show that how the inexact solution can be calculated at each iteration  for  problem \eqref{p0}. Since problem \eqref{p1} is nonconvex, we demonstrate the rate result in terms of \af{$\|{z-\mbox{prox}_{\lambda h}(z-\lambda\nabla f(z))}\|$} which is a standard termination criterion for solving constrained or composite nonconvex problems \cite{nemirovski1983problem,ghadimi2014mini,ghadimi2016accelerated}. \af{For problem \eqref{p0}, the first-order optimality condition is equivalent to find $z^*$ such that $z^*=\Pi_\Theta(z^*-\lambda\nabla f(z^*))$ for some $\lambda>0$. Hence, we show the convergence result in terms of $\epsilon$-first-order optimality condition for a vector $z$, i.e., $\|z-\Pi_\Theta(z-\lambda\nabla f(z))\|^2\leq \epsilon$.} 

\begin{algorithm}[htbp]
\caption{Inexact-proximal Accelerated Gradient Algorithm (IPAG)}
\label{alg1}
 {\bf input:} \af{$x_0,y_0 \in \mathbb R^n$, positive sequences $\{\alpha_k,\gamma_k,\lambda_k\}_k$ and Algorithm $\mathcal M$ satisfying Assumption \ref{assump2}}; \\
 {\bf for} $k=1\hdots T$ {\bf do}  \\
\mbox{(1)}\quad $z_k =(1-\alpha_k)y_{k-1} +\alpha_kx_{k-1}$;  \\
\mbox{(2)}\quad $x_k\approx \mbox{prox}_{\gamma_{k}h}\left(x_{k-1}-\gamma_k (\nabla f(z_k )+\bar \xi_k)\right)$ (solved inexactly by algorithm $\mathcal M$ with $q_k$ iterations); \\
\mbox{(3)}\quad  $y_k \approx \mbox{prox}_{\lambda_{k}h}\left(z_k -\lambda_k (\nabla f(z_k )+\bar \xi_k)\right)$ (solved inexactly by algorithm $\mathcal M$ with $p_k$ iterations);\\
{\bf end for}\\
{\bf Output:} \af{$z_N$} where $N$ is randomly selected from $\{T/2,\hdots,T\}$ with $\mbox{Prob}\{N=k\}=\frac{1}{\sum_{k=\lfloor T/2\rfloor}^T \tfrac{1-L \lambda_k}{16\lambda_k \Gamma_k}} \left(\tfrac{1-L \lambda_k}{16\lambda_k \Gamma_k}\right)$. 
\end{algorithm}

\af{\begin{assumption}\label{assump2}
For a given $c\in\mathbb R^n$ and $\gamma>0$, consider the problem $\tilde u\triangleq\mbox{prox}_{\gamma h}\left(c\right)$. 
An algorithm $\mathcal M$ with an initial point $ u_0$, output $u$ and convergence rate of $\mathcal O(1/t^2)$ within $t$ steps exists, such that $\|u-\tilde u\|^2\leq (a_1\| u_0-\tilde u\|^2+a_2)/t^2$ for some $a_1,a_2>0$. 
\end{assumption} 
Suppose  the solutions of proximal operators $\tilde x_k\triangleq\mbox{prox}_{\gamma_{k}h}\left(x_{k-1}-\gamma_k (\nabla f(z_k )+\bar \xi_k)\right)$ and $\tilde y_k \triangleq \mbox{prox}_{\lambda_{k}h}(z_k \break -\lambda_k (\nabla f(z_k )+\bar \xi_k))$ are not available exactly, instead an $e_k$-subdifferential solution $x_k$  and $\rho_k$-subdifferential solution $y_k$ are available, respectively. In particular, given $\bar \xi_k$ for the proximal subproblem in step (2) and (3) of Algorithm \ref{alg1} at iteration $k$, 
Assumption \ref{assump2} immediately implies that after $q_k$ and $p_k$ steps of Algorithm $\mathcal M$ with initial point $x_{k-1}$ and $y_{k-1}$, we have $e_k=\gamma_k(c_1\|x_{k-1}-\tilde x_k\|^2+c_2)/q_k^2$ and $\rho_k=\lambda_k(b_1\|y_{k-1} -\tilde y_k \|^2+b_2)/p_k^2$, for some $c_1,c_2,b_1,b_2>0$ where $\gamma_k, \lambda_k$ represents strong convexity of the subproblems, respectively. 
Later, in Section \ref{const opt}, we show the existence of Algorithm $\mathcal{M}$ such that it satisfies Assumption \ref{assump2}.
\begin{remark}
Note that from Assumption \ref{assump1}(iii) and \ref{assump2}, we can show the following for all $k>0$:
\begin{align}\label{bound x}
\nonumber&\|x_k-\tilde x_k\|^2\leq \tfrac{1}{q_k^2}\big[2c_1(\|x_{k-1}-\tilde x_{k-1}\|^2+\|\tilde x_{k-1}-\tilde x_k\|^2)+c_2\big]\leq \|x_{k-1}-\tilde x_{k-1}\|^2+\tfrac{8C^2+c_2}{q_k^2}\\
&\implies \|x_k-\tilde x_k\|^2\leq\|x_0-\tilde x_0\|^2+\sum_{j=1}^k \tfrac{8C^2+c_2}{q_j^2} \implies \|x_k\|\leq C+\sqrt{\|x_0-\tilde x_0\|^2+\tilde C} \triangleq B_1,
\end{align}
where $\tilde C\triangleq \sum_{j=1}^k \tfrac{8C^2+c_2}{q_j^2}$ and we used the fact that  $ \|\tilde x_k\|\leq C$. Similarly for step (3) of Algorithm \ref{alg1}, 
there exist $B_2,B_3>0$ such that the followings hold for all $k>0$,
\begin{align}\label{bound xx}
\|y_k \|\leq B_2, \qquad \|z_k \|\leq B_3. 
\end{align}
\end{remark}}
Next, we state our main lemma that provides a bridge towards driving rate statements.  
\begin{lemma}\label{main lemma}
Consider Algorithm \ref{alg1} and suppose \af{Assumption \ref{assump1} and \ref{assump2} hold} and choose stepsizes $\alpha_k$, $\gamma_k$ and $\lambda_k$ such that $\alpha_k\gamma_k\leq\lambda_k$. Let $\hat y_k\approx \mbox{prox}_{\lambda_{k}h}\left(z_k -\lambda_k \nabla f(z_k)\right)$ \aj{in the sense of \eqref{prox error} and $\hat y_k^r\triangleq \mbox{prox}_{\lambda_{k}h}\left(z_k -\lambda_k \nabla f(z_k)\right)$ for any $k\geq 1$}, then the following holds for all $T>0$.
\begin{align}\label{lem result}
&\nonumber\mathbb E[\| \hat y_N -z_N \|^2+\| \hat y_N^r -z_N \|^2] \\
&\nonumber\quad \leq \left(\sum_{k=\lfloor T/2\rfloor}^T\tfrac{1-L \lambda_k}{16\lambda_k \Gamma_k} \right)^{-1}\Big[\tfrac{\alpha_1}{2\gamma_1 \Gamma_1}\|x_0-x^*\|^2+\tfrac{\ell}{2}\sum_{k=1}^T \tfrac{\alpha_k}{\Gamma_k}\big[2B_3^2+C^2+\alpha_k(1-\alpha_k)(2B_2^2+B_1^2)\big]\\
&\quad +\sum_{k=1}^T \left(\tfrac{\lambda_k\tau^2}{\Gamma_k N_k(1-L\lambda_k)}+\tfrac{2e_k}{\Gamma_k}+\tfrac{B_1^2+C^2}{\gamma_k\Gamma_k}+\tfrac{\rho_k(1+k)}{\Gamma_k}+\tfrac{B_1^2+B_2^2}{k\lambda_k\Gamma_k}+\tfrac{\aj{5\lambda_k}\tau^2(1-L \lambda_k)}{8\Gamma_kN_k}+\tfrac{\rho_k(1-L \lambda_k)}{\lambda_k^2 \Gamma_k}\right)\Big].
\end{align}
\end{lemma}

\begin{proof}
First of all from the fact that $\nabla f(x)$ is Lipschitz, for any $k\geq1$, the following holds:
\begin{align}\label{lip}
 f(y_k )\leq f({z_k })+\langle \nabla f(z_k ),y_k -z_k \rangle +\tfrac{L}{2}\|y_k -z_k \|^2.\end{align}
Using Assumption \ref{assump1}(ii), for any $\alpha_k\in(0,1)$ one can obtain the following:
\begin{align}\label{bound}
\nonumber&f(z_k )-[(1-\alpha_k)f(y_{k-1} )+{\alpha_k} f(x)]\\
\nonumber&=\alpha_k[f(z_k {)}-f(x)]+(1-\alpha_k)[f({z_k )-f(y_{k-1} )}]\\ \nonumber&
\leq \alpha_k[\langle \nabla f(z_k ),z_k -x\rangle+\tfrac{{\ell}}{2}\|z_k -x\|^2]+(1-\alpha_k)[\langle \nabla f(z_k ),{z_k }-y_{k-1} \rangle+\tfrac{{\ell}}{2}\|x_u -y_{k-1} \|^2]\\ \nonumber&
= \langle \nabla f(z_k ),z_k -\alpha_kx-(1-\alpha_k)y_{k-1} \rangle+\tfrac{\ell\alpha_k}{2}\|z_k -x\|^2+\tfrac{\ell(1-\alpha_k)}{2}\|z_k -y_{k-1} \|^2\\ &\leq
\langle \nabla f(z_k ),z_k -\alpha_kx-(1-\alpha_k)y_{k-1} \rangle+\tfrac{\ell\alpha_k}{2}\|z_k -x\|^2+\tfrac{\ell\alpha_k^2(1-\alpha_k)}{2}\|{y_{k-1} }-x_{k-1}\|^2,
\end{align}
 where in the last inequality we used the fact that $z_k -y_{k-1} =\alpha_k(x_{k-1}-y_{k-1} )$.
From Lemma \ref{error subdef}, if $e_k$ be the error in the proximal map of update $x_k$ in Algorithm \ref{alg1}  there exists $v_k$ such that $\|v_k\|\leq \sqrt{2\gamma_k e_k}$ and $\tfrac{1}{\gamma_k}\left(x_{k-1}-x_k-\gamma_k(\nabla f(z_k )+\bar \xi_k)-v_k\right)\in {\partial}_{e_k} h(x_k)$. Therefore, from Definition \ref{subgrad}, the following holds: 
\begin{align*}
h(x)&\geq h(x_k)+\langle \tfrac{1}{\gamma_k}(x_{k-1}-x_k)-\nabla f(z_k )-\bar \xi_k-\tfrac{1}{\gamma_k}v_k, x-x_k\rangle-e_k\\
&\Longrightarrow \langle \nabla f(z_k )+\bar \xi_k,x_k-x\rangle + h(x_k)\leq h(x)-\tfrac{1}{\gamma_k}\langle v_k,x_k-x \rangle+e_k+\tfrac{1}{\gamma_k}\langle x_{k-1}-x_k,x_k-x\rangle{.}
\end{align*}
From Lemma \ref{3norms}, we have that $\tfrac{1}{\gamma_k}\langle x_{k-1}-x_k,x_k-x\rangle=\tfrac{1}{2\gamma_k}[\|x_{k-1}-x\|^2-\|x_k-x_{k-1}\|^2-\|{x_k}-x\|^2]$, therefore,
\begin{align}\label{bound1}
\nonumber&\langle \nabla f(z_k )+\bar \xi_k,x_k-{x} \rangle+h(x_k)\\ &\quad\leq h(x)-\tfrac{1}{\gamma_k}\langle v_k,x_k-x\rangle+e_k+\tfrac{1}{2\gamma_k}[\|x_{k-1}-x\|^2-\|{x_k-x_{k-1}}\|^2-\|{x_k}-x\|^2].
\end{align}
Similarly if $\rho_k$ be the error of computing the proximal map of update $y_k $ in Algorithm \ref{alg1}, then there exists $w_k$ such that  $\|w_k\|\leq \sqrt{2\lambda_k \rho_k}$ and one can obtain the following:
\begin{align}\label{b3}
\nonumber&\langle\nabla f(z_k )+\bar \xi_k,y_k -x\rangle+h(y_k )\\
&\quad\leq h(x)-\tfrac{1}{\lambda_k}\langle w_k,y_k -x\rangle+\rho_k +\tfrac{1}{2\lambda_k}[\|z_k -x\|^2-\|{y_k }-z_k \|^2-\|{y_k -x}\|^2].
\end{align}
Letting $x=\alpha_k x_k + (1-\alpha_k)y_{k-1} $ in \eqref{b3} for any $\alpha_k\geq 0$, the following holds:
\begin{align*}
&\langle\nabla f(z_k )+ \bar \xi_k,y_k -\alpha_k x_k-(1-\alpha_k)y_{k-1} \rangle+h(y_k )\\
&\quad \leq h(\alpha_k x_k+(1-\alpha_k)y_{k-1} )-\tfrac{1}{\lambda_k}\langle w_k,y_k -\alpha_k x_k-(1-\alpha_k)y_{k-1} \rangle+\rho_k\\&\qquad+\tfrac{1}{2\lambda_k}[\|{z_k }-\alpha_kx_k-(1-\alpha_k)y_{k-1} \|^2-\|y_k -z_k \|^2].
\end{align*}
From convexity of $h$ and step (1) of algorithm \ref{alg1} we obtain:
\begin{align}\label{bound2}
\nonumber &\langle\nabla f (z_k )+\bar \xi_k,{y_k }-\alpha_kx_k-(1-\alpha_k)y_{k-1} \rangle+h(y_k )\\
\nonumber &\quad \leq \alpha_k h(x_k)+(1-\alpha_k)h(y_{k-1} {)}-\tfrac{1}{\lambda_k}\langle w_k,y_k -\alpha_kx_k-(1-\alpha_k)y_{k-1} \rangle +\rho_k\\
&\qquad+\tfrac{1}{2\lambda_k}[\alpha_k^2\|x_k-x_{k-1}\|^2-\|y_k -z_k \|^2.
\end{align}
Multiplying \eqref{bound1} by $\alpha_k$ and then sum it up with \eqref{bound2} gives us the following
\begin{align}\label{sum}
\nonumber&\langle \nabla f(z_k )+\bar \xi_k,y_k -\alpha_kx-(1-\alpha_k){y_{k-1} }\rangle+h(y_k )\\
\nonumber&\quad  \leq(1-\alpha_k)h(y_{k-1} )+\alpha_k h(x)-\tfrac{\alpha_k}{2\gamma_k}[\|x_{k-1}- x\|^2-\|{x_k}-x\|^2]-\tfrac{1}{\gamma_k}\langle v_k,x_k-x\rangle\\
&\qquad+e_k+\underbrace{\tfrac{\alpha_k(\gamma_k\alpha_k-\lambda_k)}{2\gamma_k\lambda_k}}_{\text{term (a)}}\|x_k-x_{k-1}\|^2-\tfrac{1}{2\lambda_k}\|y_k -z_k \|^2-\tfrac{1}{\lambda_k}\langle w_k,{y_k }-\alpha_kx_k-(1-\alpha_k)y_{k-1} \rangle+\rho_k.
\end{align}
By choosing $\gamma_k$ such that $\alpha_k\gamma_k\leq\lambda_k$, one can easily confirm that term (a)$\leq 0$. Now combining \eqref{lip}, \eqref{bound} and \eqref{sum} and using the facts that $g(x)=f(x)+h(x)$ and $z_k =y_{k-1} +\alpha_k(x_{k-1}-y_{k-1} )$, we get the following:
\begin{align}\label{bound g}
\nonumber g(y_k )&\leq(1-\alpha_k)g(y_{k-1} )+\alpha_kg(x)-\tfrac{1}{2}(\tfrac{1}{\lambda_k}-L)\|y_k -z_k \|^2+\overbrace{\langle\bar \xi_k,\alpha_k(x-x_{k-1})+z_k -y_k \rangle}^{\text{term (b)}}\\ \nonumber&\quad+\tfrac{\alpha_k}{2\gamma_k}[\|{x_{k-1}-x}\|^2-\|x_k-x\|^2]+\tfrac{{\ell\alpha_k}}{2}\|x_{md}-x\|^2+\tfrac{{\ell\alpha_k^2}(1-\alpha_k)}{2}\|y_{k-1} -x_{k-1}\|^2\\
&\quad-\tfrac{1}{\gamma_k}\langle v_k,x_k-x\rangle+e_k-\tfrac{1}{\lambda_k}\langle w_k,{y_k }-\alpha_kx_k-(1-\alpha_k)y_{k-1} \rangle+\rho_k.
\end{align}
Moreover one can bound term (b) as follows using the Young's inequality.
\begin{align}\label{termb}
\nonumber \langle\bar \xi_k,\alpha_k(x-x_{k-1})+z_k -y_k \rangle&=\langle\bar \xi_k,\alpha_k(x-x_{k-1})\rangle+\langle\bar \xi_k,z_k -y_k \rangle\\
&\leq \langle\bar \xi_k,\alpha_k(x-x_{k-1})\rangle+\tfrac{\lambda_k}{1-L\lambda_k}\|z_k -y_k \|^2+\tfrac{1-L\lambda_k}{4\lambda_k}\|\bar \xi_k\|^2.
\end{align}
Using \eqref{termb} in \eqref{bound g}, we get the following.
\begin{align*}
 g(y_k )&\leq(1-\alpha_k)g(y_{k-1} )+\alpha_kg(x)-\tfrac{1}{4}(\tfrac{1}{\lambda_k}-L)\|y_k -z_k \|^2+\langle\bar \xi_k,\alpha_k(x-x_{k-1})\rangle+\tfrac{\lambda_k}{1-L\lambda_k}\|\bar \xi_k\|^2\\ &\quad+\tfrac{\alpha_k}{2\gamma_k}[\|{x_{k-1}-x}\|^2-\|x_k-x\|^2]+\tfrac{{\ell\alpha_k}}{2}\|x_{md}-x\|^2+\tfrac{{\ell\alpha_k^2}(1-\alpha_k)}{2}\|y_{k-1} -x_{k-1}\|^2\\
&\quad-\tfrac{1}{\gamma_k}\langle v_k,x_k-x\rangle+e_k-\tfrac{1}{\lambda_k}\langle w_k,{y_k }-\alpha_kx_k-(1-\alpha_k)y_{k-1} \rangle+\rho_k.
\end{align*}
Subtract $g(x)$ from both sides, using lemma \ref{bound gamma}, assuming ${\tfrac{\alpha_k}{\lambda_k\Gamma_k}}$ is a non-decreasing sequence and summing over $k$ from $k=1$ to T, the following can be obtained. 

\begin{align*}
&\tfrac {g(x_T )-g(x)}{\Gamma_T}+\sum_{k=1}^T \tfrac{1-L \lambda_k}{4\lambda_k \Gamma_k}\|y_k -z_k \|^2\\
&\quad \leq \tfrac{\alpha_1}{2\gamma_1 \Gamma_1}\|x_0-x\|^2-\tfrac{\alpha_{T+1}}{2\gamma_{T+1}\Gamma_{T+1}}\|x_T-x\|^2+\tfrac{\ell}{2}\sum_{k=1}^T \tfrac{\alpha_k}{\Gamma_k}\big[\|{z_k }-x\|^2+\alpha_k(1-\alpha_k)\|{y_{k-1} }-x_{k-1}\|^2\big]\\
&\qquad+\sum_{k=1}^T\tfrac{\alpha_k}{\Gamma_k} \langle\bar \xi_k,x-x_{k-1}\rangle+\sum_{k=1}^T\tfrac{\lambda_k}{\Gamma_k(1-L\lambda_k)}\|\bar \xi_k\|^2\\
&\qquad-\sum_{k=1}^T\big[\tfrac{1}{\gamma_k\Gamma_k}\langle v_k, x_k-x\rangle+\tfrac{e_k}{\Gamma_k}-\tfrac{1}{\lambda_k\Gamma_k}\langle w_k,{y_k }-\alpha_kx_k-(1-\alpha_k){y_{k-1} }\rangle+\tfrac{\rho_k}{\Gamma_k}\big].
\end{align*}
Letting $x=x^*$ and using Assumption \ref{assump1}(iii), inequalities \eqref{bound x} and \eqref{bound xx}  and the fact that $\|v_k\|\leq \sqrt{2\gamma_k e_k}$ and $\|w_k\|\leq \sqrt{2\lambda_k \rho_k}$, we can simplify the above inequality as follows:
\begin{align*}
\tfrac {g(x_T )-g(x^*)}{\Gamma_T}+\sum_{k=1}^T \tfrac{1-L \lambda_k}{4\lambda_k \Gamma_k}\|y_k -z_k \|^2 &\leq \tfrac{\alpha_1}{2\gamma_1 \Gamma_1}\|x_0-x^*\|^2+\tfrac{\ell}{2}\sum_{k=1}^T \tfrac{\alpha_k}{\Gamma_k}\big[2B_3^2+C^2+\alpha_k(1-\alpha_k)(2B_2^2+B_1^2)\big]\\
&+\sum_{k=1}^T\tfrac{\alpha_k}{\Gamma_k} \langle\bar \xi_k,x^*-x_{k-1}\rangle+\sum_{k=1}^T\tfrac{\lambda_k}{\Gamma_k(1-L\lambda_k)}\|\bar \xi_k\|^2\\
&+\sum_{k=1}^T \big(\tfrac{2e_k}{\Gamma_k}+\tfrac{B_1^2+C^2}{\gamma_k\Gamma_k}+\tfrac{\rho_k(1+k)}{\Gamma_k}+\tfrac{B_1^2+B_2^2}{k\lambda_k\Gamma_k}\big).
\end{align*}
Using the fact that $g(x_T )-g(x^*)\geq 0$, taking conditional expectation from both sides and applying Assumption \ref{assump1}(iv) on the conditional first and second moments, we get the following.
\begin{align}\label{exp}
\nonumber\sum_{k=1}^T \tfrac{1-L \lambda_k}{4\lambda_k \Gamma_k}\mathbb E[\| y_k -z_k \|^2\mid \mathcal F_k] &\leq \tfrac{\alpha_1}{2\gamma_1 \Gamma_1}\|x_0-x^*\|^2+\tfrac{\ell}{2}\sum_{k=1}^T \tfrac{\alpha_k}{\Gamma_k}\big[2B_3^2+C^2+\alpha_k(1-\alpha_k)(2B_2^2+B_1^2)\big]\\
 &+\sum_{k=1}^T\tfrac{\lambda_k\tau^2}{\Gamma_k N_k(1-L\lambda_k)}+\sum_{k=1}^T \big(\tfrac{2e_k}{\Gamma_k}+\tfrac{B_1^2+C^2}{\gamma_k\Gamma_k}+\tfrac{\rho_k(1+k)}{\Gamma_k}+\tfrac{B_1^2+B_2^2}{k\lambda_k\Gamma_k}\big).
\end{align}
\aj{To bound the left-hand side we use the following inequality by defining $y^r_k\triangleq \mbox{prox}_{\lambda_{k}h}\left(z_k -\lambda_k (\nabla f(z_k)+\bar \xi_k)\right)$ and $\hat y^r_k\triangleq \mbox{prox}_{\lambda_{k}h}\left(z_k -\lambda_k \nabla f(z_k)\right)$.  
\begin{align*}
\|y_k -z_k \|^2&={1\over 2}\|y_k -z_k \|^2+{1\over 2}\|y_k -z_k \|^2\\
&\geq {1\over 4}\|\hat y_k -z_k \|^2-{1\over 2}\|\hat y_k -y_k \|^2+{1\over 4}\|\hat y_k^r -z_k \|^2-{1\over 2}\|\hat y_k^r -y_k \|^2\\
&\geq {1\over 4}\|\hat y_k -z_k \|^2+{1\over 4}\|\hat y_k^r -z_k \|^2-{3\over 2}\|\hat y_k -\hat y_k^r \|^2-{5\over 2}\|\hat y_k^r -y_k^r \|^2-{5\over 2}\| y_k^r -y_k \|^2,
\end{align*}
where we used the fact that for any $a,b\in \mathbb R$, we have that $(a-b)^2\geq\tfrac{1}{2}a^2-b^2$ and for any $a_i\in \mathbb R$, $(\sum_{i=1}^ma_i)^2\leq m\sum_{i=1}^ma_i^2$. 
From Assumption \ref{assump1}(iv), we know that $\|\hat y^r_k -y^r_k \|^2\leq \lambda_k^2\tau^2/N_k$, also  we know that $\|\hat y_k -\hat y^r_k \|^2\leq \rho_k/\lambda_k$ and similarly $\| y_k - y^r_k \|^2\leq  \rho_k/\lambda_k$. Therefore, one can conclude that $\|y_k -z_k \|^2\geq \tfrac{1}{4}\|\hat y_k -z_k \|^2+\tfrac{1}{4}\|\hat y_k^r -z_k \|^2-\tfrac{5}{2}\lambda_k^2\tau^2/N_k-4\rho_k/\lambda_k$}. Hence, by taking another expectation from \eqref{exp} and then using this bound, the following can be obtained. 
\begin{align*}
&\nonumber\sum_{k=1}^T \tfrac{1-L \lambda_k}{16\lambda_k \Gamma_k}\mathbb E[\|\hat y_k -z_k \|^2+\|\hat y_k^r -z_k \|^2] \\
&\quad \leq \tfrac{\alpha_1}{2\gamma_1 \Gamma_1}\|x_0-x^*\|^2+\tfrac{\ell}{2}\sum_{k=1}^T \tfrac{\alpha_k}{\Gamma_k}\big[2B_3^2+C^2+\alpha_k(1-\alpha_k)(2B_2^2+B_1^2)\big]\\
&\quad +\sum_{k=1}^T \left(\tfrac{\lambda_k\tau^2}{\Gamma_k N_k(1-L\lambda_k)}+\tfrac{2e_k}{\Gamma_k}+\tfrac{B_1^2+C^2}{\gamma_k\Gamma_k}+\tfrac{\rho_k(1+k)}{\Gamma_k}+\tfrac{B_1^2+B_2^2}{k\lambda_k\Gamma_k}+\tfrac{\aj{5\lambda_k}\tau^2(1-L \lambda_k)}{8\Gamma_kN_k}+\tfrac{\rho_k(1-L \lambda_k)}{\lambda_k^2 \Gamma_k}\right).
\end{align*}
Using the fact that $\sum_{k=\lfloor T/2\rfloor}^T A_t\leq \sum_{k=1}^T A_t$ where $A_t=\tfrac{1-L \lambda_k}{16\lambda_k \Gamma_k}\mathbb E[\|\hat y_k -z_k \|^2+\|\hat y_k^r -z_k \|^2]$, dividing both side by $\sum_{k=\lfloor T/2\rfloor}^T \tfrac{1-L \lambda_k}{16\lambda_k \Gamma_k}$ and using definition of $N$ in Algorithm \ref{alg1}, the desired result can be obtained. 
\end{proof}
We are now ready to prove our main rate results. 
\begin{theorem}\label{th1}
Let $\{y_k ,x_k,z_k \}$ generated by Algorithm \ref{alg1} such that at each iteration $k\geq 1$, \aj{$e_k$-approximate solution of step (2) and $\rho_k$-approximate solution of step (3) are available through an inner algorithm $\mathcal M$}. Suppose \af{Assumption \ref{assump1} and \ref{assump2} hold} and we select the parameters in Algorithm \ref{alg1} as $\alpha_k=\tfrac{2}{k+1}$, $\gamma_k=\tfrac{k}{4L}$, $\lambda_k=\tfrac{1}{2L}$, $\Gamma_k=\tfrac{2}{k(k+1)}$ and $N_k=k+1$. Then for $B=B_1^2+B_2^2+B_3^2+C^2$, the following holds for all $T>0$.
\begin{align}\label{th result}
\mathbb E[\|\hat y_N -z_N \|^2+\|\hat y_N^r -z_N \|^2]\leq \tfrac{128}{LT^3}\left[ 2BT(T+1)\left(\tfrac{\ell}{4}+\tfrac{13\tau^2}{64LB}+4L\right)+\sum_{k=1}^T \left(\tfrac{2e_k}{\Gamma_k}+\tfrac{\rho_k(1+k)}{\Gamma_k}+\aj{\tfrac{4L^2\rho_k}{\Gamma_k}}\right)\right],
\end{align}
where $\hat y_k\approx \mbox{prox}_{\lambda_{k}h}\left(z_k -\lambda_k \nabla f(z_k)\right)$ \aj{in the sense of \eqref{prox error}, and $\hat y_k^r= \mbox{prox}_{\lambda_{k}h}\left(z_k -\lambda_k \nabla f(z_k)\right)$ for any $k\geq 1$}. 
\end{theorem}

\begin{proof}
Using the definition of $\lambda_k$ and $\Gamma_k$, we get the following. 
\begin{align}\label{sum1}
\sum_{k=\lfloor T/2\rfloor}^T \tfrac{1-L \lambda_k}{16\lambda_k \Gamma_k}=\sum_{k=\lfloor T/2\rfloor}^T \tfrac{Lk(k+1)}{32}=\tfrac{L}{32}\left[\tfrac{7T^3}{24}+T^2+\tfrac{5T}{6}\right]\geq \tfrac{LT^3}{128}.
\end{align}
\af{Next, using the definition of parameters specified in the statement of the theorem we have that}
\begin{align}\label{sum2}
\nonumber&\sum_{k=1}^T \tfrac{\alpha_k}{\Gamma_k}=\sum_{k=1}^T k=\tfrac{T(T+1)}{2}\af{,}\qquad\qquad\qquad\qquad \sum_{k=1}^T \tfrac{\tau^2}{\Gamma_kN_k}=\sum_{k=1}^T \tfrac{\tau^2k}{2}=\tfrac{\tau^2T(1+T)}{4}\af{,}\\
&\sum_{k=1}^T \tfrac{1}{\gamma_k\Gamma_k}=\sum_{k=1}^T2L(k+1)=2LT(T+3)\af{,}\qquad \sum_{k=1}^T\tfrac{1}{k\lambda_k\Gamma_k}=\sum_{k=1}^TL(k+1)=LT(T+3).
\end{align} 
Using \eqref{sum1} and \eqref{sum2} in \eqref{lem result} \af{and} the fact that $\alpha_k(1-\alpha_k)\leq 1$, $\tfrac{T+3}{T+1}\leq 2$ and defining $B=B_1^2+B_2^2+B_3^2+C^2$ we get the desired result.
\end{proof}

\begin{corollary}\label{cor1}
Let $\{y_k ,x_k,z_k \}$ \aj{be} generated by Algorithm \ref{alg1} such that at each iteration $k\geq 1$, \aj{$e_k$-approximate solution of step (2) and $\rho_k$-approximate solution of step (3) are calculated} by an inner algorithm $\mathcal M$ \aj{where} $e_k=\gamma_k(c_1\|x_{k-1}-\tilde x_k\|^2+c_2)/q_k^2$ and $\rho_k=\lambda_k(b_1\|y_{k-1} -\tilde y_k \|^2+b_2)/p_k^2$. Suppose \af{Assumptions \ref{assump1} and \ref{assump2} hold} and $p_k=k+1$ and $q_k=k$. If we choose the stepsize parameters as in Theorem \ref{th1}, then the following holds for all \aj{$T\geq 1$}.
\begin{align}\label{conv rate}
&\mathbb E[\|\hat y_N -z_N \|^2+\|\hat y_N^r -z_N \|^2]\leq\af{ \tfrac{D_1}{T}+\tfrac{D_2}{T^2}},\\ 
\nonumber & D_1\triangleq\tfrac{128}{L}\left[4B\left(\tfrac{\ell}{4}+\tfrac{13\tau^2}{64LB}+4L\right)+\left(\tfrac{c_1(2B_1^2+C^2)+c_2}{L}\right)+\left(\tfrac{b_1(2B_2^2+C^2)+b_2}{4L}\right)\right], \\& \nonumber D_2\triangleq{128}\left({b_1(2B_2^2+C^2)+b_2}\right),
\end{align}
where $\hat y_k\approx \mbox{prox}_{\lambda_{k}h}\left(z_k -\lambda_k \nabla f(z_k)\right)$ \aj{in the sense of \eqref{prox error}  and $\hat y_k^r= \mbox{prox}_{\lambda_{k}h}\left(z_k -\lambda_k \nabla f(z_k)\right)$ for any $k\geq 1$}. 
The oracle complexity (number of gradient samples) to achieve $\mathbb E[\|\hat y_N -z_N \|^2+\|\hat y_N^r -z_N \|^2]\leq \epsilon$ is $\mathcal O(1/\epsilon^2)$. 
\end{corollary}
\begin{proof}
Using the definition of the stepsizes, $p_k$, $e_k$, and $\rho_k$ one can obtain the following:
\begin{align*}
&\sum_{k=1}^T \tfrac{2e_k}{\Gamma_k}\leq \tfrac{c_1(2B_1^2+C^2)+c_2}{4L}\sum_{k=1}^T (k+1)=\left(\tfrac{c_1(2B_1^2+C^2)+c_2}{4L}\right)T(T+3).\\
&\sum_{k=1}^T \tfrac{\rho_k(1+k)}{\Gamma_k}\leq \tfrac{b_1(2B_2^2+C^2)+b_2}{4L}\sum_{k=1}^T k=\left(\tfrac{b_1(2B_2^2+C^2)+b_2}{8L}\right)T(T+1).\\
&\aj{\sum_{k=1}^T\tfrac{\rho_k}{\Gamma_k}=\left(\tfrac{b_1(2B_2^2+C^2)+b_2}{4L}\right)\sum_{k=1}^T \tfrac{k(k+1)}{(k+1)^2}\leq \left(\tfrac{b_1(2B_2^2+C^2)+b_2}{4L}\right)\sum_{k=1}^T 1=\left(\tfrac{b_1(2B_2^2+C^2)+b_2}{4L}\right)T.}
\end{align*}
Using the above inequalities in \eqref{th result}, we get the desired convergence result. Additionally, \aj{the} total number of sample gradients of the objective is $\sum_{k=1}^T N_k=\sum_{k=1}^T (k+1)=T(T+3)$ and \aj{the} total number of gradients of the constraint is $\sum_{k=1}^T p_k+q_k=\sum_{k=1}^T 2k+1=T(T+2)$. From \eqref{conv rate}, we have that $\mathbb E[\|\tilde y_N -z_N \|^2]\leq \mathcal O(1/T)=\epsilon$, hence, $\sum_{k=1}^T N_k=\mathcal O(1/\epsilon^2)$ and similarly $\sum_{k=1}^T p_k+q_k=\mathcal O(1/\epsilon^2)$. 
\end{proof}
\af{In the next corollary, we justify our choice of measure. We show that if $\mathbb E[\|\hat y_N^r -z_N \|^2]\leq\epsilon$, then the first order optimality condition \aj{for problem \eqref{p1}} holds within a ball \aj{with} radius $\sqrt \epsilon$. 

\begin{corollary}
Under the premises of Corollary \ref{cor1}, after running Algorithm \ref{alg1} for $T\geq D/\epsilon$ iterations, where $D\triangleq D_1+D_2$, the following holds. 
$$0\in \mathbb E[\nabla f(\aj{\hat y_N^r})]+\mathbb E[\partial \aj{h(\hat y_N^r)}]+\mathcal{B}\left(3L\sqrt \epsilon\right).$$\end{corollary}

\begin{proof}
Suppose $\hat y_N^r$ is a solution of $ \mbox{prox}_{\lambda_{N}h}\left(z_N -\lambda_N \nabla f(z_N)\right)$. Then $0\in \partial h(\hat y_N^r)+\nabla f(z_N)+(\hat y_N^r-z_N)/\lambda. $
Adding and subtracting $\nabla f(\hat y_N^r)$ form the right-hand side of the above inequality, gives the following:
\begin{align}\label{subgrad h}
0\in \partial h(\hat y_N^r)+\nabla f(z_N)+{1/\lambda}(\hat y_N^r-z_N)\pm \nabla f(\hat y_N^r). 
\end{align}
Moreover, using the fact that $T\geq D/\epsilon$ and $\mathbb E[\|\hat y_N^r -z_N \|^2]\leq\tfrac{D}{T}=\epsilon$ one can show the following result. 
\begin{align*}
\mathbb E\left[\|\nabla f(z_N)-\nabla f(\hat y_N^r)+1/\lambda(\hat y_N^r-z_N)\|\right]&\leq \mathbb E\left[L\|\hat y_N^r-z_N\|+1/\lambda\|\hat y_N^r-z_N\|\right]\leq   3L\sqrt {\epsilon}, 
\end{align*}
where we use the fact that  $\lambda=1/(2L)$. Using the above inequality and taking expectation from \eqref{subgrad h} the desired result can be obtained. 
\end{proof}}

In the next section, we show how Algorithm \ref{alg1} can be customized to solve problem \eqref{p0}.
\subsection{Constrained Optimization}\label{const opt}
Recall that problem \eqref{p0} can be written in a composite form using an indicator function, i.e. problem \eqref{p0} is equivalent to  $\min_x g(x)=f(x)+h(x)$, where $h(x)= \mathbb{I}_\Theta(x)$ and $\Theta=\{x\mid x\in X, \ \phi_i(x)\leq 0, \ \forall i=1,\hdots,m\}$. In step (2) and (3) of Algorithm \ref{alg1}, one needs to compute the proximal operators inexactly which are of the following form:
\begin{align}\label{co}
&\min_{u\in X} \quad {1\over 2\gamma}\left\|u- y\right\|^2\quad \mbox{s.t.} \quad \phi_i(u)\leq 0,\quad  i=1,\hdots,m, 
\end{align}
for some $y\in\mathbb{R}^n$. Problem \eqref{co} has a strongly convex objective function with convex constraints, and there has been variety of methods developed to solve such problems. One of the efficient methods for solving large-scale convex constrained optimization problem with strongly convex objective  \af{that satisfies Assumption \ref{assump2}} is first-order primal-dual scheme that guarantees a convergence rate of $\mathcal O(1/\sqrt \epsilon)$ in terms of suboptimality and infeasibility, e.g., \cite{he2015mirror,hamedani2018primal}. Next, we discuss some details of implementing such schemes as an inner algorithm for solving the subproblems in step (2) and (3) of Algorithm \ref{alg1}.

Based on Corollary \ref{cor1}, to obtain a convergence rate of $\mathcal O(1/T)$, one needs to find an $e_k$- and $\epsilon_k$-approximated solution in the sense of \eqref{prox error}. Note that since the nonsmooth part of the objective function, $h(x)$, in the proximal subproblem is an indicator function, \eqref{prox error} implies that the approximate solution of the subproblem has to be feasible, otherwise the indicator function on the left-hand side of \eqref{prox error} goes to infinity. However, the first-order primal-dual methods mentioned above find an approximate solution which might be infeasible. To remedy this issue, let $x^{\circ}$ be a slater feasible point of \eqref{co} (i.e., $\phi_i(x^{\circ})<0$ for all $i=1,\hdots,m$) and let $\hat x$ be the output of the inner algorithm $\mathcal M$ such that it is $\epsilon$-suboptimal and  $\epsilon$-infeasible, then $\tilde x=\kappa x^{\circ}+(1-\kappa)\hat x$ is a feasible point of \eqref{co} for $\kappa\triangleq \max_i \tfrac{[\phi_i(\hat x)]_+}{[\phi_i(\hat x)]_+-\phi_i(x^{\circ})}$ which is $\mathcal O(\epsilon)$-suboptimal, see the next lemma for the proof.
\begin{algorithm}[htbp]
\caption{IPAG for constrained optimization}
\label{alg2}
 {\bf input:} \af{$x^{\circ},x_0,y_0 \in \mathbb R^n$ and positive sequences $\{\alpha_k,\gamma_k,\lambda_k\}_k$, and Algorithm $\mathcal M$ satisfying Assumption \ref{assump2}}; \\
 {\bf for} $k=1\hdots T$ {\bf do}  \\
\mbox{(1)}\quad $z_k =(1-\alpha_k)y_{k-1} +\alpha_kx_{k-1}$;  \\
\mbox{(2)}\quad $x\approx \Pi_\Theta\left(x_{k-1}-\gamma_k (\nabla f(z_k )+\bar \xi_k)\right)$ (solved inexactly by algorithm $\mathcal M$ with $q_k$ iterations); \\
\mbox{(3)}\quad  $y \approx\Pi_\Theta\left(z_k -\lambda_k (\nabla f(z_k )+\bar \xi_k)\right)$  (solved inexactly by algorithm $\mathcal M$ with $p_k$ iterations);\\
\mbox{(4)}\quad $\kappa=\max_i \tfrac{[\phi_i(x)]_+}{[\phi_i(x)]_+-\phi_i(x^{\circ})}$ and $\tilde \kappa=\max_i \tfrac{[\phi_i(y)]_+}{[\phi_i(y)]_+-\phi_i(x^{\circ})}$;\\
\mbox{(5)}\quad  $ x_k=\kappa x^{\circ}+(1-\kappa) x$;\\
\mbox{(6)}\quad  $ y_k=\tilde \kappa x^{\circ}+(1-\tilde \kappa) y$;\\
{\bf end for}\\
{\bf Output:} \af{$z_N$ where $N$ is randomly selected} from $\{T/2,\hdots,T\}$ with $\mbox{Prob}\{N=k\}=\frac{1}{\sum_{k=\lfloor T/2\rfloor}^T \tfrac{1-L \lambda_k}{16\lambda_k \Gamma_k}} \left(\tfrac{1-L \lambda_k}{16\lambda_k \Gamma_k}\right)$. 
\end{algorithm}

\begin{lemma}\label{convex comb}
Let $x^{\circ}$ be a strictly feasible point of \eqref{co} and $\hat x$ be the output of an inner algorithm $\mathcal M$ such that it is $\epsilon$-suboptimal and  $\epsilon$-infeasible solution of \eqref{co}. Then $\tilde x=\kappa x^{\circ}+(1-\kappa)\hat x$ is a feasible point of \eqref{co} and  an $\mathcal O(\epsilon)$-approximate solution in the sense of \eqref{prox error} 
where $\kappa= \max_i \tfrac{[\phi_i(\hat x)]_+}{[\phi_i(\hat x)]_+-\phi_i(x^{\circ})}$. 
\end{lemma}

\begin{proof}
Let $x^*$ be the optimal solution of \eqref{co}. Since $\hat x$ is $\epsilon$-suboptimal and  $\epsilon$-infeasible solution, $\hat x\in X$ and the following holds:
\begin{align*}
\big|\tfrac{1}{2\gamma}\|\hat x-y\|^2-\tfrac{1}{2\gamma}\|x^*-y\|\big|\leq \epsilon,\quad \mbox{and}\quad [\phi_i(\hat x)]_+\leq \epsilon, \ \forall i\in\{1,\hdots,m\}.
\end{align*} 
Since $X$ is a convex set and $x^{\circ},\hat x\in X$, then clearly $\kappa x^{\circ}+(1-\kappa)\hat x\in X$ for any $\kappa\in[0,1]$. Moreover, $\phi_i(x^\circ)<0$ for all $i$, hence $\kappa= \max_i \tfrac{[\phi_i(\hat x)]_+}{[\phi_i(\hat x)]_+-\phi_i(x^{\circ})}\in [0,1]$ and $\kappa\leq \tfrac{\epsilon}{\min_i\{-\phi_i(x^{\circ})\}}$. From convexity of $\phi_i(\cdot)$, one can show the following for all $i=1,\hdots,m$.
\begin{align*}
\phi_i(\tilde x)\leq \kappa\phi_i(x^{\circ})+(1-\kappa)\phi_i(\hat x)\leq 0,
\end{align*} 
where we used the definition of $\kappa$. Hence, $\tilde x$ is a feasible point of \eqref{co}. Next, we verify $\tilde x$ satisfies \eqref{prox error}. 
\begin{align*}
&\tfrac{1}{2\gamma}\|\tilde x-y\|^2+\mathbb I_\Theta(\tilde x)-\tfrac{1}{2\gamma}\|x^*-y\|^2-\mathbb I_\Theta(x^*)\\
&\quad =\tfrac{1}{2\gamma}\|\tilde x-y\pm x^{\circ}\|^2-\tfrac{1}{2\gamma}\|x^*-y\|^2\\
&\quad \leq \tfrac{\kappa^2}{2\gamma}\|x^\circ-y\|^2+\tfrac{(1-\kappa)^2}{2\gamma}\|\hat x-y\|^2+\tfrac{\kappa(1-\kappa)}{\gamma}\|x^\circ-y\|^2\|\hat x-y\|^2-\tfrac{1}{2\gamma}\|x^*-y\|^2\\
&\quad = \tfrac{\kappa^2}{2\gamma}\|x^\circ-y\|^2+\tfrac{\kappa(1-\kappa)}{\gamma}\|x^\circ-y\|^2\|\hat x-y\|^2+(1-\kappa^2)\left[\tfrac{1}{2\gamma}\|\hat x-y\|-\tfrac{1}{2\gamma}\|x^*-y\|\right]\\
&\qquad-\tfrac{1-(1-\kappa^2)}{2\gamma}\|x^*-y\|^2\\
&\quad \leq \tfrac{\kappa^2}{2\gamma}\|x^\circ-y\|^2+\tfrac{\kappa(1-\kappa)}{\gamma}\|x^\circ-y\|^2\|\hat x-y\|^2+\epsilon\leq \mathcal O(\epsilon),
\end{align*}
 where we used the fact that $\hat x, x^*$ are feasible, $\hat x$ is $\epsilon$-suboptimal and $\kappa\leq \tfrac{\epsilon}{\min_i\{-\phi_i(x^{\circ})\}}$.
 \end{proof}
 In the following corollary, we show that the output of Algorithm \ref{alg2} is feasible to problem \eqref{p0} and satisfies $\epsilon$-first-order optimality condition.  
\begin{corollary}\label{cor4}
Consider problem \eqref{p0}. Suppose \af{Assumption \ref{assump1} and \ref{assump2} hold} and let $\{y_k ,x_k,z_k \}$ \aj{be} generated by Algorithm \ref{alg2} such that the stepsizes and parameters are chosen as in Corollary \ref{cor1}. Then the iterates are feasible and $\mathbb E\left[\|z_N-\Pi_\Theta\left(z_N-\lambda_N\nabla f(z_N)\right)\|^2\right]\leq \mathcal O(\epsilon)$ holds with an oracle complexity $\mathcal O(1/\epsilon^2)$. 
\end{corollary}
\begin{proof}
From Lemma \ref{convex comb} we know that the iterates are feasible and from Corollary \ref{cor1}, we conclude that $\mathbb E[\hat \|y_N^r-z_N\|^2]\leq \epsilon$ with an oracle complexity $\mathcal O(1/\epsilon^2)$. Considering problem \eqref{p0}, definition of $\hat y_N^r$ is equivalent to $\hat y_N^r=\Pi_\Theta\left(z_N-\lambda_N\nabla f(z_N)\right)$ which implies the desired result.
\end{proof}

\section{NUMERICAL EXPERIMENTS}\label{numer}
The goal of this section is to present some computational results to compare the performance of the IPAG method with another competitive scheme.  \aj{For Algorithm \ref{alg2}, we consider accelerated primal-dual algorithm with backtracking (APDB)} method introduced by \cite{hamedani2018primal} as the inner algorithm $\mathcal M$. \aj{In particular, APDB is a primal-dual scheme with a convergence guarantee of \af{$\mathcal O(1/T^2)$} in terms of suboptimality and infeasibility when implemented for solving \eqref{co} which satisfies the requirements of Corollary \ref{cor4}, i.e., produces approximate solutions for the proximal subproblems.}

\indent {\bf Example.} The IPAG method is benchmarked against the inexact constrained proximal point algorithm (ICPP) introduced by \cite{boob2019stochastic}.  Consider the following \aj{stochastic} quadratic programming problem:
\begin{align*}
\min_{-10\leq x\leq 10} \ &f(x)\triangleq-\tfrac{\epsilon}{2}\|DBx\|^2+ \tfrac{\tau}{2}\mathbb E[\|Ax-b(\xi)\|^2]\\
\text{s.t.} \quad  &\ \tfrac{1}{2}x^T Q_{i} x+d_{i}^Tx-c_i \leq 0, \quad \forall i=1\hdots m,
\end{align*} 
where $ \aj{A} \in \mathbb R^{p\times n}$, $p=n/2$, $B \in \mathbb R^{n\times n}, D \in \mathbb R^{n\times n}$ is a diagonal matrix, $b(\xi)= b+\omega \in \mathbb R^{p\times 1}$, where the elements of $\omega$  have an i.i.d. \aj{standard normal distribution}. 
\mb{ The entries of matrices $A$, $B$, and vector $b$ are generated by  sampling from the uniform distribution ${U}$[0,1] and the diagonal entries of matrix $D$ are generated by sampling from the discrete uniform distribution ${U}$\{1,1000\}}.
Moreover, $(\delta, \tau) \in \mathbb R_{++}^2$ , $ Q_i \in \mathbb R^{n\times n}$, $ d_i \in \mathbb R^{n\times 1}$ and $ c_i \in \mathbb R$ for all $i\in\{1,\hdots,m\}$. We chose scalers $\delta$ and $\tau$ such that $\lambda_{min}(\nabla^2f)<0$, i.e., minimum eigenvalue of the Hessian is negative. Note that Assumption \ref{assump1}(i) holds for $x^{\circ}=\mathbf 0$, where $\mathbf 0$ is the vector of zeros. 

\begin{table}[htb]\scriptsize
	\centering
\begin{tabular}{|c|c|c|c|c|c|c|c|}
\hline
\multicolumn{2}{|c|}{} & \multicolumn{3}{c|}{IPAG}  & \multicolumn{3}{c|}{ICPP}   \\ \hline
n          & m         & $f(x_T)$      & Infeas. & CPU(s) &  $f(x_T)$        & Infeas. & CPU(s) \\ \hline
100        & 25        & -6.78e+5 & 0      & 12.10  & -4.85e+4 & 3.56e-1 & 32.99  \\ \hline
100        & 50        & -8.53e+5 & 0      & 31.76  & -2.42e+4 & 3.23e-1 & 65.79  \\ \hline
100        & 75        & -4.18e+5 & 0      & 52.43  & -2.16e+4 & 3.75e-1 & 110.53 \\ \hline
200        & 25        & -3.22e+6 & 0      & 65.56  & -1.81e+5 & 2.56e-1 & 132.18 \\ \hline
200        & 50        & -1.85e+6 & 0      & 90.49  & -8.45e+4 & 4.54e-1 & 208.84 \\ \hline
200        & 75        & -1.33e+6 & 0      & 138.75 & -7.78e+4 & 3.93e-1 & 287.20 \\ \hline
\end{tabular}
\begin{tabular}{|c|c|c|c|c|}
\hline
\multicolumn{3}{|c|}{} &{IPAG}  &{ICPP}   \\ \hline
n          & m      &std.   & $f(x_T)$  &  $f(x_T)$       \\ \hline
100        & 25       &1&-6.7866e+5 & -4.8563e+4 \\ \hline
100        & 25        &5&-6.5288e+5&-4.8596e+4 \\ \hline
100        & 25        & 10&-6.2336e+5&-4.8528e+4 \\ \hline
200        & 50        & 1& -1.8552e+6 & -8.4550e+4\\ \hline
200        & 50        & 5&-1.8452e+6&-8.5264e+4 \\ \hline
200        & 50        &10&-1.8383e+6&-8.6096e+4 \\ \hline
\end{tabular}\caption{Comparing IPAG and ICPP.}
\label{tab1}
\end{table}

In Table \ref{tab1} (left), we compared the objective value, CPU time, and infeasibility \af{(Infeas.)} of our proposed method with ICPP \cite{boob2019stochastic} and in Table \ref{tab1} (right) we compared \aj{the} methods for different choices of standard deviation (std.) \af{of $\omega$}. To have a fair comparison, we fixed the oracle complexity (i.e. the number of computed gradients is equal for both methods). As it can be seen in the table, for different choices of $m$ and $n$, IPAG scheme outperforms ICPP. For instance, when we have 25 constraints and $n=100$, the objective value for our scheme \aj{reaches} $f(x_T)=-6.78e+5 $ which is significantly smaller than $-4.85e+4$ for ICPP method. Note that our scheme, in contrast to ICPP, obtains a feasible solution \aj{at} each iteration. Similar behavior can be \aj{observed} for different choices of the standard deviation \af{in Table \ref{tab1} (right)}. 
 standard deviation.

\footnotesize

\bibliographystyle{spmpsci}    

\bibliography{biblio}

\begin{thebibliography}{10}
\providecommand{\url}[1]{{#1}}
\providecommand{\urlprefix}{URL }
\expandafter\ifx\csname urlstyle\endcsname\relax
  \providecommand{\doi}[1]{DOI~\discretionary{}{}{}#1}\else
  \providecommand{\doi}{DOI~\discretionary{}{}{}\begingroup
  \urlstyle{rm}\Url}\fi

\bibitem{basu2019optimal}
Basu, K., Nandy, P.: Optimal convergence for stochastic optimization with
  multiple expectation constraints.
\newblock arXiv preprint arXiv:1906.03401  (2019)

\bibitem{boob2019stochastic}
Boob, D., Deng, Q., Lan, G.: Stochastic first-order methods for convex and
  nonconvex functional constrained optimization.
\newblock arXiv preprint arXiv:1908.02734  (2019)

\bibitem{ghadimi2013stochastic}
Ghadimi, S., Lan, G.: Stochastic first-and zeroth-order methods for nonconvex
  stochastic programming.
\newblock SIAM Journal on Optimization \textbf{23}(4), 2341--2368 (2013)

\bibitem{ghadimi2016accelerated}
Ghadimi, S., Lan, G.: Accelerated gradient methods for nonconvex nonlinear and
  stochastic programming.
\newblock Mathematical Programming \textbf{156}(1-2), 59--99 (2016)

\bibitem{ghadimi2014mini}
Ghadimi, S., Lan, G., Zhang, H.: Mini-batch stochastic approximation methods
  for constrained nonconvex stochastic programming.
\newblock Manuscript, Department of Industrial and Systems Engineering,
  University of Florida, Gainesville, FL \textbf{32611} (2014)

\bibitem{hamedani2018primal}
Hamedani, E.Y., Aybat, N.S.: A primal-dual algorithm with line search for
  general convex-concave saddle point problems.
\newblock SIAM Journal on Optimization \textbf{31}(2), 1299--1329 (2021)

\bibitem{he2015mirror}
He, N., Juditsky, A., Nemirovski, A.: Mirror prox algorithm for multi-term
  composite minimization and semi-separable problems.
\newblock Computational Optimization and Applications \textbf{61}(2), 275--319
  (2015)

\bibitem{kong2019complexity}
Kong, W., Melo, J.G., Monteiro, R.D.: Complexity of a quadratic penalty
  accelerated inexact proximal point method for solving linearly constrained
  nonconvex composite programs.
\newblock SIAM Journal on Optimization \textbf{29}(4), 2566--2593 (2019)

\bibitem{lan2019accelerated}
Lan, G., Yang, Y.: Accelerated stochastic algorithms for nonconvex finite-sum
  and multiblock optimization.
\newblock SIAM Journal on Optimization \textbf{29}(4), 2753--2784 (2019)

\bibitem{lan2016algorithms}
Lan, G., Zhou, Z.: Algorithms for stochastic optimization with expectation
  constraints.
\newblock arXiv preprint arXiv:1604.03887  (2016)

\bibitem{li2021rate}
Li, Z., Chen, P.Y., Liu, S., Lu, S., Xu, Y.: Rate-improved inexact augmented
  lagrangian method for constrained nonconvex optimization.
\newblock In: International Conference on Artificial Intelligence and
  Statistics, pp. 2170--2178. PMLR (2021)

\bibitem{li2020augmented}
Li, Z., Xu, Y.: Augmented lagrangian based first-order methods for convex and
  nonconvex programs: nonergodic convergence and iteration complexity.
\newblock arXiv preprint arXiv:2003.08880  (2020)

\bibitem{lin2019inexact}
Lin, Q., Ma, R., Xu, Y.: Inexact proximal-point penalty methods for constrained
  non-convex optimization.
\newblock arXiv preprint arXiv:1908.11518  (2019)

\bibitem{ma2017demand}
Ma, K., Bai, Y., Yang, J., Yu, Y., Yang, Q.: Demand-side energy management
  based on nonconvex optimization in smart grid.
\newblock Energies \textbf{10}(10), 1538 (2017)

\bibitem{nemirovski1983problem}
Nemirovski, A.S., Yudin, D.B.: Problem complexity and method efficiency in
  optimization  (1983)

\bibitem{pham2020proxsarah}
Pham, N.H., Nguyen, L.M., Phan, D.T., Tran-Dinh, Q.: Proxsarah: An efficient
  algorithmic framework for stochastic composite nonconvex optimization.
\newblock Journal of Machine Learning Research \textbf{21}(110), 1--48 (2020)

\bibitem{schmidt2011convergence}
Schmidt, M., Roux, N.L., Bach, F.: Convergence rates of inexact
  proximal-gradient methods for convex optimization.
\newblock In: Proceedings of the 24th International Conference on Neural
  Information Processing Systems, pp. 1458--1466 (2011)

\bibitem{tran2014primal}
Tran-Dinh, Q., Cevher, V.: A primal-dual algorithmic framework for constrained
  convex minimization.
\newblock arXiv preprint arXiv:1406.5403  (2014)

\bibitem{ullah2020applications}
Ullah, Z., Al-Turjman, F., Mostarda, L., Gagliardi, R.: Applications of
  artificial intelligence and machine learning in smart cities.
\newblock Computer Communications \textbf{154}, 313--323 (2020)

\bibitem{xu2019iteration}
Xu, Y.: Iteration complexity of inexact augmented lagrangian methods for
  constrained convex programming.
\newblock Mathematical Programming pp. 1--46 (2019)

\bibitem{zhang2018convergence}
Zhang, S., He, N.: On the convergence rate of stochastic mirror descent for
  nonsmooth nonconvex optimization.
\newblock arXiv preprint arXiv:1806.04781  (2018)

\end{thebibliography}

\end{document}